\documentclass{article}

\usepackage{amsmath}
\usepackage{amsthm}
\usepackage{amssymb}

\newtheorem{theorem}{Theorem}
\newtheorem{lemma}{Lemma}
\newtheorem{prop}{Proposition}
\newtheorem{corollary}{Corollary}
\theoremstyle{definition}
\newtheorem{remark}{Remark}
\newcommand{\la}{\lambda}

\newcommand{\F}{{\Bbb F}}

\newcommand{\N}{{\Bbb N}}

\newcommand{\eps}{\varepsilon}

 \title{Paley's Theory for Lacunary Fourier Series on Discrete Groups: a Semigroup-Interpretation}
\author{Tao Mei  \footnote{Research partially supported by the NSF grant DMS-1700171.} }

\begin{document}

\maketitle
  \begin{abstract} This note interprets Paley's theory for lacunary Fourier series using semigroup-BMO and $H^1$ spaces. This interpretation allows an extension of Paley's theory  to general discrete groups, complementing the results of Rudin (\cite{Ru62})  for abelian groups with a total order, and Lust-Piquard and Pisier's work (\cite{LP91}) for  lacunary Fourier series with operator-valued coefficients.
  
\end{abstract}
 
\section*{Introduction}

Denote by ${\Bbb T}$ the unit circle. Given a lacunary sequence $  (j_k)_{ k\in{\Bbb N}}\in {\Bbb Z}$, i.e. $$\frac {|j_{k+1}|}{j_k}>1+\delta$$ for some $\delta>0$, the classical Khintchine's inequality says that 
\begin{eqnarray*}
(\sum_k |c_k|^2)^\frac12\simeq ^{c_\delta} \|\sum_k c_kz^{j_k}\|_{L^1({\Bbb T})} .
\end{eqnarray*}
This shows that $\ell_2$ embeds into $L^1$. However, the projection $$P: f \mapsto \hat f(j_k)$$ is NOT bounded from $L^1({\Bbb T})$ to $\ell_2$. Here $\hat f$ denotes for the Fourier transform of $f$. 
This can be easily seen by looking at the so-called Riesz products.
 Paley's theory is an improvement of  Khintchine's inequality. It says that,
$$(\sum_k |c_k|^2)^\frac12\simeq ^{c_\delta}\inf \{\|f\|_{L^1}; f\in L^1({\Bbb T}), {\rm supp} \hat f\subset {\Bbb N}, \hat f(j_k)=c_k,\forall k\in{\Bbb N}\}.$$
This shows that  the projection $P$ is bounded from the analytic $L^1$ to $\ell_2$, which has important applications, e.g. to Grothendieck's theory on 1-summing maps.

 Let $H^1({\Bbb T})$ be the real Hardy space on the unite circle, that consists of integrable functions  such that both their analytic and the anti-analytic parts   are integrable. %Let $E=\{j_k, k\in{\Bbb N}\}\subset{\Bbb Z}$.
 Paley's theory says that 
  \begin{eqnarray}\label{paley1}
  (\sum_k |c_k|^2)^\frac12\simeq ^{c_\delta}\inf \{\|f\|_{H^1}; f\in H^1({\Bbb T}), \hat f(k)=c_k,\forall k\in E\},
  \end{eqnarray}
  for $E=\{j_k, k\in{\Bbb N}\}\subset{\Bbb Z}$.
Let us call $E\subset {\Bbb N}$ a Paley set if the above equivalence holds for all $(c_k)_k\in \ell_2$.   Rudin proved that $E$ is a Paley set only if $$\sup_{n\in{\Bbb N}} \# E\cap [2^n,2^{n+1}]<C$$ which is equivalent to say that $E$ is a finite union of lacunary sequences. 
  
  By Fefferman-Stein's  $H^1$-BMO duality theory, (\ref{paley}) has an equivalent formulation that, for any $c_k\in \ell_2$,
 \begin{eqnarray}\label{paley2}
 (\sum_k |c_k|^2)^\frac12\simeq ^{c_\delta}\|\sum_k c_kz^{j_k}\|_{BMO({\Bbb T})}.
  \end{eqnarray}
  Here BMO denotes the bounded mean oscillation  (semi)norm
    $$\|g\|_{BMO}=\sup _I \frac1{|I|}\int_I |g-g_I|$$ with the supremum taking for all arc $I\in {\Bbb T}$.

 This article is  an interpretation of Paley's theory in the semigroup language and an extension  to non-abelian discrete groups. 
 Let $P_t, t>0,$ to be the Poisson integral operator that sends $e^{ik\theta}$ to $r^{|k|}e^{ik\theta}$ with $r=e^{-t}$. Here is an equivalent characterization of the  classical BMO and $H^1$-norms by $P_t$'s. That, for $f\in L^1({\Bbb T}$),
 \begin{eqnarray*}
 \|f\|_{BMO}\simeq \sup_t\|(P_t|f-P_tf|^2)\|_{L^\infty({\Bbb T})}^\frac12\\
 \|f\|_{H^1}\simeq \|(\int_0^\infty|\partial P_tf|^2tdt)^\frac12\|_{L^1({\Bbb T})}^\frac12.
 \end{eqnarray*}

Throughout this article, we consider a discrete group $G$ and a conditionally negative   length  $\psi$ on $G$.  That is to say $\psi$ is a ${\Bbb R}_+$-valued function on $G$ satisfying $\psi(g)=0$ iff $g=e$, $\psi(g)=\psi(g^{-1})$,  and
\begin{eqnarray}\label{CN}
  \sum_{g,h}\overline{a_g}a_h\psi(g^{-1}h)\leq0
\end{eqnarray}
for any finite collection of coefficients $a_g\in {\Bbb C}$ with $\sum_g a_g=0$. 
We say a sequence $h_k\in G$ is $\psi$-lacunary if there exists a constant $\delta>0$ such that
\begin{eqnarray*}
   \psi(h_{k+1}) &\geq& (1 + \delta){\psi(h_k)}  \\
 \psi(h_k^{-1}h_{k'})&\geq& \delta\max\{\psi(h_k),\psi(h_{k'})\}.
 \end{eqnarray*} 
 for any $k,k'$.   Note the second condition follows from the first one if we require $\psi$ is sub-additive, i.e. $\psi(hg)\leq C\psi(h)+\psi(g)$. Let $\la$ be the regular left representation of $G$. We say $$x=\sum_k c_k\la_{h_k}$$ is a $\psi$-lacunary Fourier series if the sequence $h_k$ is $\psi$-lacunary. We say $ x $ is a  lacunary Fourier series if there is a conditionally negative $\psi$ so that  $h_k$ is $\psi$-lacunary. 

%When $G={\Bbb Z}$, and $\psi(k)=|k|, k\in \Z$. Paley and Kochneff/Sagher/Zhou (\cite{KSZ90}) prove that,  for any $\psi$-lacunary Fourier series $x=\sum_k c_k\la_{h_k}\in L^2(\T)$, we have 
%\begin{eqnarray}\label{KSZ}
%\|x\|^2_{BMO}\simeq \sum_k|c_k|^2.
%\end{eqnarray}  By interpolation, this implies that every  lacunary Fourier series has an equivalent $L^p$ and $L^2$ norm, which is a   fundamental theory in the classical Fourier analysis. 
%We will show that Kochneff/Sagher/Zhou' result extends to non-abelian discrete groups by considering semigroup BMO associated with $\psi$, while an analogue of Rudin's theorem on the size of Sidon sets (thus  Leli\`evre's theorem on BMO)  fails for $\F_2$.
  
   Let $$T_t:\la_g\mapsto e^{-t\psi(g)}\la_g$$ be the semigroup associated with $\psi$. We will show that, 
   
 \vskip.5cm
 \noindent
   {\bf Main Theorem.} Assume $(h_k)$ is a $\psi$-lacunary sequence. Then, for any sequence $c_k\in B(H)$,
 \begin{eqnarray*}
 \|\sum_k c_k\la_{h_k}\|_{BMO_c(\psi)}^2&\simeq^{ c_\delta }&\|\sum_k |c_k|^2 \|.\\
\inf \{\|x\|_{H^1_c(\psi)}; \hat x(h_k)=c_k,\forall k\in{\Bbb N}\}&\simeq^{ c_\delta}& tr (\sum_k |c_k|^2)^\frac12.
 \end{eqnarray*}
  
\noindent  Here the semigroup-$H^1$ and BMO-norms are defined as
  \begin{eqnarray*} \|x\|_{H^1_c(\psi)}=tr\otimes\tau (\int_0^\infty |\frac{\partial_s T_sx}{\partial s}|^2sds)^\frac12
 \\
   \|x\|_{{\rm BMO}_c(\psi)}=\sup_s \|T_s|x-T_sx|^2\|^\frac12.
  \end{eqnarray*}
  with $tr, \tau$   the canonical traces on $B(H)$ and the reduced $C^*$ algebra of $G$. One gets the $L^p$ estimate for all $1<p<\infty$ by interpolation.
\section{BMO estimate.}

Given a discrete group $G$, we denote by $({\cal L}(G),\tau)$ the group von Neumann algebra with its canonical trace $\tau$. Denote by $L^p(\hat G)$ the
associated noncommutative $L^p$ spaces, that is  the closure of ${\cal L}(G)$ w.r.t. the norm $\|x\|_p=(\tau |x|^p)^\frac1p$.  If $G$ is abelian, then ${  L}^p(\hat G)$ is the canonical $L^p$ space of functions on the dual group $\hat G$. In particular, if $G={\Bbb Z}$, then $\lambda_k=e^{ikt}, k\in {\Bbb Z}$ and ${  L}^p(\hat{\Bbb Z})=L^p({\Bbb T})$, the space of all $p$-integrable functions  on the unit circle. Please refer to \cite{PX03} for details on noncommutative $L^p$ spaces.

Given a conditionally negative length   $\psi$ on $G$,
Schoenberg's
theorem says that $$T_t: \lambda_g=e^{-t\psi(g)}\lambda_g$$ extends to a symmetric Markov semigroup of operators on the group von Neumann algebra $L^p(\hat G), 1\leq p\leq \infty$.
Following  \cite{JM12} and \cite{M08}, let us set
\begin{eqnarray}
\|x\|_{\mathrm{BMO}_c(\psi ) }&=&\sup_{0<t<\infty}\| T_{t }|x-T_{t}x|^2\| ^\frac12, \label{BMOT}
\end{eqnarray}
for $x\in L^2(\hat G)$. Let ${\mathrm{BMO}(\psi ) }$ be the space of all $ x\in L^2(\hat G)$ such that
\begin{eqnarray}
  \|x\|_{BMO(\psi)}=\max\{\|x\|_{ BMO_c(\psi)},\|x^*\|_{BMO_c(\psi)}\}<\infty.
\end{eqnarray}

 \begin{lemma} \label {JM12}([JM12]) We have the following interpolation result 
$$ [BMO(\psi),L^1(\hat G)]_{\frac{1}{p}} = L^p(\hat G) $$
for
 $1<p<\infty$.     
\end{lemma}

\begin{lemma}\label{CS}
For $a_s\in {\Bbb R}_+$, $c_s,b_s\in B(H)$, we have, for any $0<p,q,r<\infty, \frac1p+\frac1q=\frac1r$,
\begin{eqnarray}
\|\sum_s a_sc^*_sb_s\|&\leq &\|\sum_s |c_s|^2 a_s\|^\frac12\|\sum_s |b_s|^2 a_s\|^\frac12\\
\|\sum_k a_sc^*_sb_s\|_{L^r}&\leq& \|\sum_s |c_s|^2 a_s\|_{L^p}^\frac12\|\sum_s |b_s|^2 a_s\|_{L^q}^\frac12.
\end{eqnarray}
\end{lemma}
\begin{proof} This is simply the Cauchy-Schwartz inequality.
\end{proof}
  
 \begin{theorem}
 Assume $(h_k)$ is a $\psi$-lacunary sequence. Then, for any $x=\sum_kc_k\la _  {h_k} $, we have
 \begin{eqnarray}
 \|x\|_{BMO(\psi)}^2\simeq^{ c_\delta}\max\{\|\sum_k |c_k|^2 \|,\|\sum_k |c_k^*|^2 \|\}.
 \end{eqnarray}
 \end{theorem}
\begin{proof} An easy calculation shows that
\begin{eqnarray*}
 T_t|x -T_tx|^2
  = \sum _{k,j}  a_{k,j}(c_k\la_{h_k})^*c_j\la_{h_j},
 \end{eqnarray*}
with $$a_{k,j}= e^{-t\psi(h^{-1}_kh_j )}(1-e^{-t\psi(h_k^{-1})})(1-e^{-t\psi(h_j)})\geq0.$$
By the lacunary property $\psi(h_k^{-1}h_j)\geq |\psi(h_k)-\psi(h_j)|$, we have
\begin{eqnarray*}
  \sum _{k}  a_{k,j}
  %&\leq &\sum_{k}( e^{-t|\psi(h_k)-\psi(h_j)|}-e^{-t\psi(h_k)-t\psi(h_j^{-1})})\\
  &\leq &\sum _{t\psi(h_k)\leq 1} (1-e^{-t\psi(h_k^{-1})})    +\sum _{t\psi(h_k)> 1}   e^{-t\psi(h^{-1}_kh_j )} \\ 
  &\leq &\sum _{t\psi(h_k)\leq 1}  t  \psi(h_k)   +\sum _{t\psi(h_k)> 1}    e^{- t \delta\psi(h_k)}  \\
  &\leq & 1+\delta^{-1}+\frac1{1-e^{-  {\delta^2} }}
   \leq c _\delta.  
 \end{eqnarray*}
We then get $\sup_j \sum_k a_{k,j}\leq c_\delta$. Similarly, $\sup_k \sum_j a_{k,j}\leq c_\delta.$  By Lemma \ref{CS}, we have
\begin{eqnarray*}
 \|T_t|x-T_tx|^2\|
  &\leq &\|\sum _{k,j} |c_k|^2a_{k,j}\|^\frac12 \|\sum_{k,j}|c_j|^  2a_{k,j}\|^\frac12\\
    &\leq&c_\delta  \|\sum_k |c_k|^2 \|.
 \end{eqnarray*}
 Taking supremum on $t$, we get $\|x\|^2_{BMO_c }\leq c_\delta \|\sum_k |c_k|^2 \|.$ Taking the adjoint, we prove the upper estimate. The lower estimate is obvious by taking conditional expectation $\tau$ and sending $t$ to $\infty$.
 \end{proof}

%\medskip
Given a length-lacunary sequence $h_k\in G$,  define the linear map $T$ from $L^\infty(\ell_2)$ to $BMO$ by 
 $$T((c_k))=\sum_k c_k\la_{h_k}.$$
 Then $T$ has a norm $c_\delta$ from $L^\infty(\ell_2)$ to $BMO$ and norm 1 from  $L^2(\ell_2)$ to $L^2(\hat G).$
By the interpolation result Lemma \ref{JM12}, we get 
\begin{corollary} 
 Assume $(h_k)$ is a $\psi$-lacunary sequence for some conditionally negative $\psi$. We have that, for any $p>2, x=\sum_kc_k\la _  {h_k} $, 
\begin{eqnarray}
 \|x\|_{p}^2\leq c^{\frac {p-2}p}_\delta p^2\max\{\|\sum_k |c_k|^2 \|_{\frac p2},\|\sum_k |c_k^*|^2 \|_{\frac p2}\}.
 \end{eqnarray}
 By duality, we get, for any $1<p<2$,
 \begin{eqnarray}
 \inf\{\|\sum_k |a_k|^2 \|_{\frac p2}+\|\sum_k |b_k^*|^2 \|_{\frac p2};c_k=a_k+b_k\} \lesssim \inf\{\|x\|^2_p; \hat x(h_k)=c_k\}.\label{10}
 \end{eqnarray}
 \end{corollary}
 
 \begin{remark} Corollary 1 is proved in \cite{JMX06} page 118 with a worse constant.%independently proved in \cite{JMPX} by using noncommutative Riesz transforms associated with semigroups.
    If, $G=\F_n,\psi $ is the reduced word length, it is also easy to verify that $\psi$-lacunary set is $B(2)$ in the sense of W. Rudin, so it is a $\Lambda_4$ set by Harcharras's work\cite{Ha99}. This does not seem clear  for $B(p)$ with $p>2$.   
 \end{remark}
 \begin{remark}  The sequence of free generators $\{g_i,i\in {\Bbb N}\}$ of $\F_\infty$ is a $\psi$-lacunary sequence for some $\psi$. Indeed, let $\pi$ be the group homomorphism on $\F_\infty$ sending $g_i$ to $g_{i}^{2^i}$. Then $\psi(h)=|\pi(h)|$ is a conditionally negative function.
 \end{remark}

   \begin{remark}  One can extend (\ref{10})  to the range $0<p\leq1$ as a Khintchine-type inequality
    \begin{eqnarray}
 \inf\{\|\sum_k |a_k|^2 \|_{\frac p2}+\|\sum_k |b_k^*|^2 \|_{\frac p2};c_k=a_k+b_k\} \lesssim \|\sum_k c_k\la_{h_k} \|_{\frac p2}^2.
 \end{eqnarray}
   following Piser-Ricard's argument \cite{PR17}. For the $p=1$ case, one may follow Haagerup-Musat's argument to get a better constant.
\end {remark}

We will prove a column version of (\ref{10}) in the next section.
  \section{$H^1$-Estimate}

   \begin{lemma}\label{Lemma}
   Let $x=\sum_k c_k\la_{h_k}\in L^2(\hat G)$, Then, we have
 \begin{eqnarray}
 \tau( \int_0^\infty|\partial_s T_sx|^2sds)^\frac12&\leq &\frac12 (\sum_k |c_k|^2 )^\frac12. \label{H1<}
 \end{eqnarray}
Moreover, if we assume $(h_k)$ is a $\psi$-lacunary sequence, then
\begin{eqnarray}
  \| \int_0^\infty|\partial_s T_sx|^2sds\|&\leq &c_\delta\|\sum_k |c_k|^2  \|.
 \end{eqnarray}
 \end{lemma}

\begin{proof} An elementary calculation shows that
\begin{eqnarray*}
 \int_0^\infty|\partial_s T_sx|^2sds&=&\sum _{k,j}   (c_k\la_{h_k})^*c_j\la_{h_j}\psi(h_i)\psi(h_j)\int_0^\infty e^{-s(\psi(h_k)+\psi(h_j))}sds\\
  &=& \sum _{k,j}  a_{k,j}(c_k\la_{h_k})^*c_j\la_{h_j},
 \end{eqnarray*}
with $$a_{k,j}= \frac {\psi(h_k)\psi(h_j)}{(\psi(h_k)+\psi(h_j))^2}\geq0$$ since $ \int_0^\infty e^{-t}tdt=1$.
So
\begin{eqnarray*}
   \tau(\int_0^\infty|\partial_s T_sx|^2sds)^\frac12 &\leq& (\tau\int_0^\infty|\partial_s T_sx|^2sds)^\frac12\\
  &= &(\sum_{k}|c_k|^  2a_{k,k})^\frac12=\frac12  (\sum_k |c_k|^2)^\frac12.
   \end{eqnarray*}
On the other hand, it is easy to see that 
\begin{eqnarray*}
  \sup_j \sum_k a_{k,j}\leq c_\delta, \ \ \sup_k \sum_j a_{k,j}\leq c_\delta.
\end{eqnarray*}
 Applying Lemma 2 for $p=q=\infty$, we have \begin{eqnarray*}
   \|\int_0^\infty|\partial_s T_sx|^2sds\|
  &\leq &\|(\sum _{k,j} |c_k|^2a_{k,j})^\frac12\| \|(\sum_{k,j}|c_j|^  2a_{k,j})^\frac12\|\\
  &\leq& c_\delta  \|\sum_k |c_k|^2\|.
 \end{eqnarray*}

 \end{proof}
 
   \begin{theorem}
 Assume $(h_k)$ is a $\psi$-lacunary sequence. Then, we have
 \begin{eqnarray}
 tr (\sum_k |c_k|^2 )^\frac12\simeq^{ c_\delta}\inf\{\tau\otimes tr( \int_0^\infty|\partial_s T_sx|^2s)^\frac12 ; \tau (x \la^*_{ h_k})=c_k\}.
 \end{eqnarray}
 \end{theorem}

\begin{proof}  
By duality, we may choose $b_k$ such that $\|\sum |b_k|^2\|=1$ and
\begin{eqnarray*}tr (\sum_k |c_k|^2 )^\frac12= tr\sum c^*_kb_k=(\tau\otimes tr) x^*y,
  \end{eqnarray*} 
with $y=\sum b_k\la_{h_k}$ and any (finite) Fourier sum $x$ such that $\tau (x\la^*_{h_k})=c_k$.
We then have
\begin{eqnarray*}
(\tau\otimes tr)x^*y&=&4\tau\otimes tr \int_0^\infty \partial_s T_sx^*\partial_s T_sy sds\\
 &\leq& 4\tau\otimes tr (\int_0^\infty|\partial_s T_sx|^2sds)^\frac12\|\int_0^\infty|\partial_s T_sy|^2sds\|^\frac12
\end{eqnarray*}
 Combining the above estimates with Lemma \ref{Lemma}, we obtain \begin{eqnarray*}
 tr (\sum_k |c_k|^2 )^\frac12\leq 4c_\delta \tau\otimes tr(\int_0^\infty|\partial_s T_sx|^2sds)^\frac12.
    \end{eqnarray*}
    The other direction follows by taking $tr$ on both sides of (\ref{H1<}).
 \end{proof}
\section{Large $\Lambda_\infty$ sets on $\F_2$}

We call a subset $A\in G$ is completely Sidon, if $\{\lambda_h, h\in A\}$ is  completely unconditional in ${\mathcal L}(\hat G)$, i.e. there exists a constant $C_A$ such that
$$\|\sum_{h_k\in A} \eps_k c_k\lambda_{h_k}\|\leq C_A\|\sum_{h_k\in A}  c_k\lambda_{h_k}\|,$$
for any  choice $\eps_k=\pm$, $c_k\in B(H)$.
We call a subset $A\in G$ is completely $\Lambda_\infty$, if there exists a constant $C_A$ such that
\begin{eqnarray}\label{lainfty}
\|\sum_{h_k\in A}  c_k\lambda_{h_k}\|\leq C_A\max\{\|\sum_{h_k\in A}  |c_k|^2\|^\frac12,\|\sum_{h_k\in A}  |c_k^*|^2\|^\frac12\},\end{eqnarray}
for any  choice of finite many $c_k\in B(H)$. 
We say $A$ is completely $\Lambda_{bmo,\psi}$ if  we take  the BMO$(\psi)$-norm on  the left hand side of (\ref{lainfty}).
Obviously, a completely $\Lambda_\infty$ set is completely $\Lambda_{bmo,\psi}$ for any $\psi$, and is completely Sidon.

%When $G={\Bbb Z}$, this is equivalent to say $\|\sum_{h_j\in A} a_j \lambda_h_j\|_{{\cal L}({\Bbb Z})}\simeq \sum |a_j|$
  Let ${\mathcal P}_d$ (${\mathcal P}_{\leq d}$)   be the collection of all reduced words of $\F_n$ with   length $=d$ (${\leq d}$).
  When $G={\Bbb Z}$,  a   classical theory of Rudin says that, for any Sidon set $A$ of $\F_1$, we have $\# (A\cap {\mathcal P}_{\leq d})\lesssim \log \# {\mathcal P}_{\leq d}$, and  
  Leli\`evre (\cite{Le97}) prove that every $\Lambda_{bmo,|\cdot|}$  is a finite combination of Hadamard lacunary sets, thus a Sidon set.

Fix a generating set $S=\{g_k,k\in{\Bbb Z}_*\}$ of $\F_\infty$, with the convention that $g_k^{-1}=g_{-k}$. Let ${\mathcal Q}_n\subset \F_\infty$ be the collection of symmetric words of length $2n$, 
 $${\mathcal Q}_n=\{g_{k_1}g_{k_2}\cdots g_{k_n} g_{k_n}\cdots g_{k_2}g_{k_1}; |g_{k_j}|=1, g_{k_j}\neq g_{k_{j+1}}^{-1}, k_j\in {\Bbb Z}_*\}.$$
 
 The following Proposition is the key observation for our example. We include a proof although this  maybe obvious for  experts. 

 \begin{prop} ${\mathcal Q}_n$ is   a free subset of $\F_\infty$.
 \end{prop}
\begin{proof} Let us first introduce a few notations. Given a reduced word $h\in \F_\infty$, we denote by $L_h$ the subset of all reduced words $g$ that start with $h$, that is $L_h=\{g\in\F_\infty; |g|\geq|h|, |{h^{-1}}g|=|g|-|h|\}$.
Suppose $h\in {\mathcal P}_{2n}$, denote by $h^l,h^r$ the left half  and   the right half of $h$, i.e. the reduced words in ${\mathcal P}_{n}$ such that $h=h^lh^r$.  
 We will use the fact, that  the condition  $|hg|> |g|$ holds iff $g \notin L_{ (h^r)^{-1}}$  and implies that $hg\in L_{h^l}$. 
 
 Given any $m$ elements  $x_j\in {\mathcal Q}_n, 1\leq j\leq m $ such that $x^{-1}_k\neq x_{k+1}$ for any $1\leq k<m$, it is obvious that  $|x_2x_1 |>|x_1 | $.    Assume  $|x_j \cdots x_2x_1|>|x_{j-1} \cdots x_2x_{1}|$. That is $|x_jg|> |g|$ for $g=x_{j-1} \cdots x_2x_{1}$. Then $x_jg\in L_{x_j^l}$. So $g'=x_jg\notin L_{(x_{j+1}^r)^{-1}}$ since $x_j\neq x_{j+1}^{-1}$. Then $|x_{j+1}g'|>|g'|$. We then get $|x_j\cdots x_2x_1|>|x_{j-1} \cdots x_2x_{1}|$ for all $1<j\leq m$ by induction. Therefore, $x_m \cdots x_2x_1 \neq e$.
We then conclude that ${\mathcal Q}_n$ is a free set.  
  \end{proof}
   Now, for $x=\sum_{h\in {\mathcal Q}_n} c_h\lambda_h $ with $h\in{\mathcal Q}_n$ and $c_h\in B(H)$, we have by Haagerup and Pisier's inequality (\cite{HaPi93})that
\begin{eqnarray}\label{HaO}\|x\|\leq  2\max\{   \|\tau |x|^2\|^\frac12, \|\tau|x^*|^2\|^\frac12 \}. \label{HaPi}
 \end{eqnarray}

%\begin{corollary} Suppose $A_k\subset L_k\cap R_{\phi(k)}\cap {\mathcal P}^b_d$ are completely $\Lambda_\infty$ (Sidon) sets with $\sup_kC_{A_k}<\infty$. Then
%$A=\cup_k A_k$ is a completely $\Lambda_\infty$ (Sidon) set with constant $\leq 81\sup_kC_{A_k}$.
%\end{corollary}

 \noindent{\bf Example.}  
Let $\pi$ be  the group homomorphism from $\F_\infty$ into $\F_2$ with free generators $a,b$, such that
 $$\pi (g_k)=a^{k}ba^{-k}, k\in{\Bbb N}.$$
 By (\ref{HaPi}),  $\pi({\mathcal Q}_n)$ is a complete $\Lambda_\infty$ set of $\F_2$ for each $n\in \N$. Therefore, it is completely Sidon, and is completely $\Lambda_{bmo(|\cdot|)}$ with $|\cdot|$ the word length on $\F_2$. However,  $\pi({\mathcal Q}_n)$  is not a finite union of  $|\cdot|$-lacunary set, contrary to Leli\'evre's theorem  for  $\F_1$,
%We can pick $\pi(A)=\{a^k_1ba^{-k}b^{3^j}a^k ba^{-k};  j,k\in \N\}$ and
%$$\pi(A^n)=\{a^{k_n}ba^{-k_n} \cdots a^{k_1}ba^{-k_1}b^{3^j }a^{k_n}ba^{-k_n}\cdots a^{k_1}ba^{-k_1}; j,k_i\in \N\}.$$
In fact, it is easy to see that  
%$\# (\pi(A)\cap {\mathcal P}_{\leq 3^j})\simeq 3^j$.
$\# (\pi({\mathcal Q}_n)\cap {\mathcal P}_{\leq 2nm})\simeq m^n$ while 
$\log \# {\mathcal P}_{\leq 2nm}\simeq nm$ as $m\rightarrow \infty$. 

%$A_k=\{g_kg_jg_k,j\in \N\}$ is Sidon with constant 2.
%$A=\{g_kg_jg_k,j\in \N\}$ is Sidon with constant $\leq 18$.
\begin{remark} Let $\phi$ be  an injection  from ${\Bbb N}  $ to ${\Bbb N} $. Let  $\phi (k)=-\phi (-k )$ for $k<0$.   (\ref{HaPi})  holds for 
$${\mathcal Q}_n=\{g_{k_1}g_{k_2}\cdots g_{k_n} g_{\phi (k_{n})}\cdots g_{\phi(k_2)}g_{\phi(k_1)}; |g_{k_j}|=1, g_{k_j}\neq g_{k_{j+1}}^{-1} \} $$ as well.
\end{remark}
\begin{remark} Suppose $A$ is a completely Sidon set. Does there exist a conditionally negative $\psi$, so that $\#(A\cap {\cal P}^\psi_n)\leq \log \#  {\cal P}^\psi_n$?
\end{remark}
\bibliographystyle{amsplain}

\vskip30pt

%\hfill \noindent \textbf{Marius Junge} \\
%\null \hfill Department of Mathematics
%\\ \null \hfill University of Illinois at Urbana-Champaign \\
%\null \hfill 1409 W. Green St. Urbana, IL 61891. USA \\
%\null \hfill\texttt{junge@math.uiuc.edu}

%\
\bigskip
\hfill \noindent \textbf{Tao Mei} \\
\null \hfill Department of Mathematics
\\ \null \hfill Baylor University \\
\null \hfill One bear place, Waco, TX  USA \\
\null \hfill\texttt{tao\_mei@baylor.edu}

%\hfill \noindent \textbf{Javier Parcet} \\
%\null \hfill Instituto de Ciencias Matem{\'a}ticas \\ \null \hfill
%CSIC-UAM-UC3M-UCM \\ \null \hfill Consejo Superior de
%Investigaciones Cient{\'\i}ficas \\ \null \hfill C/ Nicol\'as Cabrera 13-15.
%28049, Madrid. Spain \\ \null \hfill\texttt{javier.parcet@icmat.es}
\end{document}